\documentclass[12pt]{amsart}

 \usepackage{todonotes, graphicx, hyperref, verbatim, enumerate}
 
\newtheorem{theorem}[equation]{Theorem}
\newtheorem{prop}[equation]{Proposition}
\newtheorem{lemma}[equation]{Lemma}
\newtheorem{cor}[equation]{Corollary}
\newtheorem{corollary}[equation]{Corollary}

\theoremstyle{remark}
\newtheorem{remark}[equation]{Remark}
\theoremstyle{definition}

\numberwithin{equation}{section}

\newcommand{\mes}{{\rm mes}}

\newcommand{\bO}{{\mathbf{O}}}

\newcommand{\fH}{{\mathfrak{H}}}
\newcommand{\fI}{{\mathfrak{I}}}

\newcommand{\integers}{{\bf Z}}

\newcommand{\ratls}{{\bf Q}}
\newcommand{\reals}{{\bf R}}

\newcommand{\bd}{{\bf d}}

\newcommand{\brc}{{\bar c}}
\newcommand{\brC}{{\bar C}}
\newcommand{\brg}{{\bar g}}

\newcommand{\brgamma}{{\bar\gamma}}
\newcommand{\brdelta}{{\bar\delta}}

\newcommand{\cK}{{\mathcal{K}}}
\newcommand{\cM}{{X}}

\newcommand{\cA}{{\mathcal{A}}}
\newcommand{\hP}{{\widehat{P}}}
\newcommand{\heps}{{\widehat{\eps}}}

\newcommand{\tD}{{\widetilde{d}}}
\newcommand{\tK}{{\widetilde{K}}}
\newcommand{\tL}{{\widetilde{L}}}
\newcommand{\tP}{{\widetilde{P}}}
\newcommand{\teps}{{\widetilde{\eps}}}

\newcommand{\cH}{{\mathcal{H}}}
\newcommand{\hyp}{{\mathcal{H}}}

\newcommand{\tg}{{\tilde g}}
\newcommand{\tgamma}{{\tilde\gamma}}
\newcommand{\tsigma}{{\tilde\sigma}}

\newcommand{\naturals}{\mathbb{N}}

\newcommand{\tr}{{\rm tr}}

\newcommand{\LL}{\mathbb{L}}

\newcommand{\eps}{{\varepsilon}}
\newcommand{\brx}{{\bar x}}

\begin{document}
\title[Small gaps]
{On small gaps in the length spectrum}

\author[D. Dolgopyat]{Dmitry Dolgopyat}
\address{Department of Mathematics, 
Mathematics Building, 
University of Maryland, 
College Park, MD 20742-4015, USA} 
\email{dmitry@math.umd.edu}

\author[D. Jakobson]{Dmitry Jakobson}
\address{Department of Mathematics and
Statistics, McGill University, 805 Sherbrooke Str. West, Montr\'eal
QC H3A 2K6, Ca\-na\-da.} \email{jakobson@math.mcgill.ca}

\thanks{D. D. was supported by NSF.
D. J. was partially supported by NSERC, FQRNT and Peter Redpath fellowship}

\subjclass[2000]{Primary 37C25, 53C22; Secondary 20H10, 37C20, 37D20, 53D25}

\begin{abstract}
We discuss upper and lower bounds for the size of gaps in the length spectrum of negatively curved manifolds.
For manifolds with algebraic generators for the fundamental group, we establish the existence of exponential
lower bounds for the gaps.  On the other hand, we show that the existence of arbitrary small gaps is topologically generic: this is established both for surfaces of constant negative curvature 
(Theorem \ref{ThTop-HS}), and for the space of negatively curved metrics 
(Theorem \ref{thm:smallgapneg}).  While arbitrary small gaps are topologically generic, it is plausible
that the gaps are not too small for almost every metric. One result in this direction is presented in 
Section 5.
\end{abstract}

\maketitle


\section{Introduction: geodesic length separation in negative curvature}

On negatively curved manifolds, the number of closed geodesics of length $\leq T$ 
grows exponentially in $T.$ 
(We refer the reader to \cite{MS, PP, PS} for a comprehensive discussion about 
the growth and
distribution of closed geodesics).

The abundance of closed geodesics leads to the natural question about the 
sizes of gaps in the length spectrum.
In the current note we present a number of results related to this question. 
In some situations we are able to control
the gaps from below, while in other we show that such control is not possible 
in general.

We note that a presence of exponentially large multiplicities in the length spectrum 
of a Riemannian manifold (which can be considered as a limiting case of small 
gaps) changes the level spacings distribution of Laplace eigenvalues on that manifold, 
see e.g. \cite{Luo:Sarnak}.   

For generic Riemannian metrics, the length spectrum is {\em simple} 
\cite{Abr,Anosov-generic}, so for 
any closed geodesic $\gamma$, only $\gamma^{-1}$ will have the same length.
So, by the Dirichlet 
box principle, there exist {\em exponentially small} gaps between the lengths of different 
geodesics.

Accordingly, it seems interesting to investigate manifolds where the gaps between the 
lengths of different geodesics have exponential {\em lower bound:} there exist 
constants $C,\beta>0$, such that for any $l_1\neq l_2\in {\rm
Lsp}(M)$ (length spectrum of the negatively curved manifold $M$), we have
\begin{equation}\label{separation}
|l_2-l_1|>Ce^{-\beta\cdot\max(l_1,l_2)}.
\end{equation}

This assumption
is satisfied for arithmetic hyperbolic groups by the trace separation criterion 
(cf. \cite{Takeuchi} and \cite[\S 18]{Hej}).
In Section~\ref{ScDiop2D}
we explain (see Theorem \ref{ThAlgGen}) why the assumption \eqref{separation}
holds for hyperbolic manifolds whose fundamental group has algebraic elements.

In particular, the surfaces satisfying \eqref{separation} form a dense set in the corresponding
Teichmuller space. 
On the other hand the existence of arbitrary small gaps is {\em topologically generic} 
as is shown in
Theorem \ref{ThTop-HS} for surfaces of constant negative curvature and in 
Theorem \ref{thm:smallgapneg} for the space of negatively curved metrics endowed 
with $C^r$-topology, for any $r>0.$

While arbitrary small gaps are topologically generic, it is plausible that the gaps 
are not too small
for almost every metric. One result in this direction is presented in 
Section \ref{ScH2-KR} there we obtain an explicit
lower bound for the gaps valid for almost every hyperbolic surface.

Length separation between closed geodesics is relevant for the study of 
wave trace formulas on 
negatively-curved manifolds: to accurately study contributions from exponentially 
many closed geodesics to the 
wave trace formula, it is necessary to separate contributions from geodesics which differ
either on the length axis, or in phase space. We remark that a suitable version 
of \eqref{separation} always holds in {\em phase space}: small tubular 
neighbourhoods of closed geodesics 
in phase space are disjoint, as shown in \cite{JPT}.  
Since there exist metrics for which the size of the length gaps 
cannot be controlled (Theorem \ref{thm:smallgapneg}), the authors 
in \cite{JPT} established microlocal 
wave trace formula, and used the separation of closed trajectories in 
phase space in the proof.  

\section{Diophantine results for hyperbolic manifolds.}
\label{ScDiop2D}

\subsection{Distances between algebraic numbers.}
In this section we consider gaps in the length spectrum for manifolds whose fundamental group
admits algebraic generators. But first we provide a few general results about the algebraic numbers.

\begin{lemma}
\label{LmSmallestRoot} 
If $\alpha$ is a root of $P(x)=x^D+a_{D-1} x^{D-1}+\dots+a_0$ then
$$ |\alpha|\geq \frac{|a_0|}{\left(1+\sum_{j=0}^{D-1} |a_j|\right)^{D-1}}. $$
\end{lemma}

\begin{proof}
Let $\alpha_j$ be the roots of $P$ counted with multiplicities. 
We claim that $|\alpha_j|\leq R:=1+\sum_j {|a_j|}.$ Indeed if $|x|>R$
then since $R>1$ we get
$$  |P(x)|\geq |x|^D-\sum_{j=0}^{D-1} |a_j| |x|^j\geq 
|x|^{D-1}\left(|x|-\sum_{j=0}^{D-1} |a_j|\right)>0. $$
The result follows since $\prod_j |\alpha_j|=|a_0|.$
\end{proof}

Given a field $K$ which is an extension of $\ratls$ of degree $d$ let $\fH(L, N, p)$ be the set of all elements of $K$ of the
form $\frac{\beta}{N^p}$ where $\beta\in\bO_K$ and for each automorphism $\sigma_j$ of $K$ we have
$|\sigma_j(\beta)|\leq L.$

\begin{lemma}
\label{LmAuto}
If $0\neq \alpha\in \fH(L, N, p)$ then
$$ |\alpha|\geq \frac{1}{L^{d-1} N^p}. $$
\end{lemma}

\begin{proof}
Indeed $|\beta| L^d\geq 1$ because $\displaystyle \prod_{j=1}^d |\sigma_j(\beta)|\geq 1. $
\end{proof}

Let $\fI(L, N, p, D)$ be the set of numbers which satisfy
$$ \alpha^{E}+a_{E-1} \alpha^{E-1}+\dots+a_0=0 $$
where $E\leq D$ and $a_j\in \fH(L,N,p).$

\begin{corollary}
\label{CrAlgAlg}
If $0\neq \alpha\in \fI(L, N, p, D)$ then 
$$ |\alpha|\geq \frac{1}{L^{d-1} N^p (DL+1)^{D-1}}. $$
\end{corollary}
\begin{proof}
Since $\alpha\neq 0$ we can assume after possibly reducing the degree of the polynomial that
$a_0\neq 0.$ Then the result follows by combining Lemmas \ref{LmSmallestRoot} 
and \ref{LmAuto}
\end{proof}

\begin{prop} 
\label{PrSumAlg}
(see e.g. \cite[Section 5.8]{vW})
There exists constants $C$ and $q$ such that
if $\alpha_1, \alpha_2\in \fI(L, N, p, D)$ then
$\alpha_1+\alpha_2$ and $\alpha_1-\alpha_2$ are in
$\fI(CL^q, N, pq, D^2).$
\end{prop}

Combining Proposition \ref{PrSumAlg} with Corollary \ref{CrAlgAlg} we obtain

\begin{corollary}
\label{CrAlgDif}
If $\alpha_1, \alpha_2\in \fI(L, N, p, d)$ then either
$\alpha_1=\alpha_2$ or
$$ |\alpha_1-\alpha_2|\geq \frac{c}{L^{q(d-1)} N^{pq} L^{D^2}}. $$
\end{corollary}

\subsection{Manifolds with algebraic generators of $\pi_1$}
We now formulate the main result of this section.
\begin{theorem}
\label{ThAlgGen}
Let $\cM$ be a hyperbolic manifold such that the generators of $\pi_1(\cM)$ belong to
$PSO_{n,1}(\bar{\ratls}).$ Then \eqref{separation} holds.
\end{theorem}

We remark that in dimension 2 groups 
satisfying the assumptions of
Theorem \ref{ThAlgGen} form a dense set in the corresponding
Teichmuller space $T_g$. This can be established, for example, by the arguments of Section~\ref{ScH2-KR}.

If $n\geq 3$ then \cite[Theorem 0.11]{GR} building on earlier results of of Selberg \cite{Selberg} and Mostow (\cite{Mostow}) shows
the conditions of Theorem \ref{ThAlgGen} are satisfied for all finite volume hyperbolic manifolds.
Hence we obtain
\begin{corollary}
\eqref{separation} holds for finite volume hyperbolic manifolds of dimension $n\geq 3.$
\end{corollary}

The proof of Theorem \ref{ThAlgGen} is similar to the proof of 
Proposition 3 in \cite{GJS}, where it is shown that the
rotation matrices in ${\rm SU}(2)\cap M_2(\overline{\ratls})$
satisfy the {\em Diophantine condition} defined in \cite{GJS}.  
Related results for other Lie groups were established in \cite{ABRS,Br11,Varju}.  
Related questions were also discussed in \cite{Glu}.  

\subsection{Proof of Theorem \ref{ThAlgGen}}
\begin{proof}
Let $\gamma_1$ and $\gamma_2$ be two closed geodesics. Let $l_j$ be the length of $\gamma_j,$
$W_j$ be the word fixing $\gamma_j,$ $B_j$ be the matrix corresponding to $W_j,$
$m_j$ be the word length of $W_j$ and $r_j=l_j/2.$ To establish 
\eqref{separation} it suffices to show that
\begin{equation}
\label{SeparationExp}
\left| e^{r_1}-e^{r_2}\right| \geq \brC e^{-\brc \max(r_1, r_2)}. 
\end{equation}

Without a loss of generality we assume that $m_j\gg 1.$
By (\cite[Lemma 2]{Milnor}) we know that
\begin{equation}
\label{EqWL-GL}
\frac{l_j}{C} \leq m_j \leq C l_j 
\end{equation}
so \eqref{separation} if trivial unless $m_j$ and $m_2$ are comparable. Let us assume to fix our ideas
that $m_2\geq m_1.$ By assumption there is a finite extension $K$ of $\ratls$ and numbers 
$L$ and $N$ such that all entries of the generators belong to 
$\fH(L, N, 1).$ Accordingly the entries of $B_j$ belong to 
$$\fH((L(n+1))^{m_j}, N, m_j).$$
Closed geodesics on $\cM$ correspond to {\em loxodromic} elements of 
$\pi_1(\cM)\subset PSO_{n,1}$ (also called {\em boosts}) 
that that fix no points in $\hyp^n$ and fix two points in $\partial\hyp^n$.
It is shown in the proof of \cite[Thm. I.5.1]{Franchi} that $B_j$ has precisely two positive 
real eigenvalues $\alpha_{1,j}=e^{r_j}$ and $\alpha_{2,j}=e^{-r_j}$; all other eigenvalues of 
$B_j$ have modulus one. Since the coefficients of the characteristic polynomial of $B_j$
are the sums of minors we have
$$ e^{r_j}\in \fI((L(n+1))^{(n+1)m_j} (n+1)!, N, m_j (n+1)). $$
Reducing to the common denominator we see that both $e^{r_1}$ and $e^{r_2}$ belong to
$$ \fI((L(n+1))^{(n+1)m_2} N^{m_2-m_1} (n+1)!, N, m_2 (n+1)). $$
Now \eqref{SeparationExp} follows by Corollary \ref{CrAlgDif} and \eqref{EqWL-GL}.
\end{proof}

\begin{remark}
In dimension two the proof can be simplified slightly by remarking that 
$2\cosh(l_j/2)=\tr B_j\in K$.  
An alternative proof of Theorem \ref{ThAlgGen} could proceed by 
using explicit formulas for the lengths of closed geodesics on hyperbolic manifolds 
(see e.g. \cite[(3), p. 246]{Prasad:Rapinchuk}) and the estimates for linear forms 
in logarithms (see e.g. \cite[Chapter 2]{BW}).
The proof we give is more elementary, using only basic facts about 
algebraic numbers and matrix eigenvalues; and fairly concrete.  
\end{remark}



\section{Small gaps for surface of constant negative curvature.}
Let $$G_g=\{(A_1, \dots A_{2g})\in (SL_2(\reals))^{2g}: 
[A_1, A_2] [A_3, A_4]\dots [A_{2g-1}, A_{2g}]=I\}. $$
\begin{theorem}
\label{ThTop-HS}
The set of tuples $(A_1, A_2\dots A_{2g})\in G_g$
where \eqref{separation} fails is topologically generic.
\end{theorem}

\begin{proof}
Let $\gamma_A$ denote the closed geodesic whose lift to the fundamental cover joins $q$ and $A q.$ 
Let $\LL$ denote the length spectrum of the geodesics $\gamma_A$ where $A$ belongs to a subgroup
generated by $A_1$ and $A_2.$ Note that for a dense set of tuples it holds that for each $\delta$
there exists $L$ such that for $l>L$ the set $[l, l+\delta]$ intersects $\LL.$ 
One way to see this is to consider the
geodesics $\gamma_{A_1^k A_2^m}.$ Their length have asymptotics
$$ \kappa(A_1, A_2) k \lambda_1+m \lambda_2 $$
where $e^{\lambda_j}$ is the leading eigenvalue of $A_j.$ 
Note that for typical $A_1, A_2$ we have $\kappa(A_1, A_2)\neq 0$ and $\lambda_1$ and $\lambda_2$ 
 are non commensurable. Consider a geodesic $\brgamma=\gamma_{A_3 A_1^n}$ where 
$n$ is very large. By the foregoing discussion there exists $l\in \LL$ such that 
$|l-L_\brgamma|<\delta.$ Now consider the perturbations of $A_3$ of the form
$ A_3(\eta)=\left(\begin{array}{cc} 1 & \eta \\
                                                                  0 & 1 
                                      \end{array}\right) A_3 . $    
Assume that 
$A_3 A_1^n=\left(\begin{array}{cc} a & b \\
                                                                  c & d 
                                      \end{array}\right)  .$  
After applying a small perturbation if necessary we can assume that all entries of this matrix have the same                                                               
order as its trace. Then
$$ \tr(A_3(\eta) A_1^n)=\tr(A_3 A_1^n)+\eta c, $$
so by a small perturbation we can make $L_{\gamma_{A_3(\eta) A_1^n}}$ as close to
$l$ as we wish. Now the result follows by a standard Baire category argument (cf. Section \ref{SSVarCurvSG}). 
\end{proof}


\section{Constructing metrics with small gaps in the length spectrum}
\label{SSVarCurvSG}
This section is devoted to the proof of the following fact.
\begin{theorem}\label{thm:smallgapneg}
For any $r>3$ for any negatively-curved $C^r$ metric $g$, 
for any function $F(t)$ (which we assume is monotone and
fast decreasing), and a number $\delta>0$, there exists a metric
$\brg$, such that $||\brg-g||_{C^r}<\delta$ 
and there exists an infinite sequence of pairs of
closed $\tg$-geodesics $\gamma_{1,j},\gamma_{2,j}$ with
$L_\brg(\gamma_{i,j})\to\infty$ as $j\to\infty$, and
\begin{equation}\label{smallgap}
|L_\brg(\gamma_{1,j})-L_\brg(\gamma_{2,j})|<\min
\{F(L_\brg(\gamma_{1,j})),F(L_\brg(\gamma_{2,j}))\}.
\end{equation}
\end{theorem}
This shows that, in general, one {\em cannot} obtain good lower
bounds for gaps in the length spectrum for 
a $C^r$ open set of negatively curved metrics.

Theorem \ref{thm:smallgapneg} follows from the lemma below by a standard Baire category argument.

\begin{lemma}
\label{LmEqLen}
Given a metric $g$ and numbers $L$ and $\delta$ there is a metric
$\tg$ such that $||g-\tg||_{C^r}\leq \delta$ and there are two $\tg$-geodesics $\gamma_1$ and $\gamma_2$ such that
$$L_\tg(\gamma_1)=L_\tg(\gamma_2)>L.$$
\end{lemma}
We also need the following fact
\begin{lemma}
\label{LmComp}
Let $g$ and $\tg$ be two negatively curved metrics such that $||g-\tg||_\infty\leq \delta$ and 
$\gamma$ and $\tgamma$ be two closed geodesics for $g$ and $\tg$ respectively of lengths 
$L$ and $\tL$. If $\gamma$ and $\tgamma$ are homotopic then
$$ \frac{L}{1+\delta}\leq\tL\leq L(1+\delta). $$
\end{lemma}

\begin{proof}
Recall that for negatively curved there exists a unique geodesic in each homotopy class and this geodesic is length
minimizing. The second inequality follows since the length of $\gamma$ with respect
to $\tg$ is at most $L(1+\delta)$ and $\tgamma$ is shorter. The first inequality follows from the second by interchanging 
the roles of $g$ and $\tg.$
\end{proof}

\begin{proof}[Proof of Theorem \ref{thm:smallgapneg}]
We claim that given metric $g$ and numbers $k\in \naturals$ and $\delta>0$ there exists a metric $\brg$ such that
$||\brg-g||_{C^r}<\delta$ and for each $j=1\dots k$ there are geodesics $\gamma_{1,j}, \gamma_{2,j}$ such that
\begin{equation}
\label{EqKGaps}
 L_\brg(\gamma_{i,j})>j, \quad |L_\brg(\gamma_{1,j})-L_\brg(\gamma_{2,j})|\leq 
F(\max(L_\brg(\gamma_{1,j}), L_\brg(\gamma_{2,j}))). 
\end{equation}
It follows that the space of metrics satisfying \eqref{smallgap} is topologically generic and hence dense.

It remains to construct $\brg$ satisfying \eqref{EqKGaps}. 

By Lemma \ref{LmEqLen} we can find $g_1$ such that $||g-g_1||_{C^r}<\frac{\delta}{2}$
and there are two geodesics $\gamma_{1,1}$ and $\gamma_{2,1}$ such that
$$  L_{g_1}(\gamma_{i,1})>1 \text{ and } L_{g_1}(\gamma_{1,j})=L_{g_1}(\gamma_{2,j}). $$
For $j\geq 1$ we apply Lemma \ref{LmEqLen} to find $g_j$ such that
$$ ||g_j-g_{j-1}||_{C^r} \leq\min\left(\frac{\delta}{2^j}, \min_{l=1}^{j-1} 
\frac{F(L_{g_l}(\gamma_{1,l}))+1)}{L_{g_l}(\gamma_{1,l})  2^{j-l+1}}\right)  $$
and there are two geodesics $\gamma_{1,j}, \gamma_{2,j}$ such that
$$ L_{g_j}(\gamma_{1,j})=L_{g_j}(\gamma_{2,j})>j. $$
Then $g_k$ satisfies the required properties since, by Lemma \ref{LmComp}, 
the lengths of $g_{i,l}$ have changed by less than
$F(L_{g_l}(\gamma_{1, l})+1)/2$ in the process of making consecutive inductive steps.
\end{proof}
\begin{remark}
In particular if we continue the above procedure for the infinite number of steps then the limiting metric
will satisfy the conditions of Theorem \ref{thm:smallgapneg}.
\end{remark}

The proof of Lemma \ref{LmEqLen} relies on two facts. 
If $\gamma$ is a closed geodesic let $\nu_\gamma$ denote the invariant measure for the geodesic flow supported on
$\gamma.$ Let $h$ denote the topological entropy of the geodesic flow. Let $\mu$ denote the Bowen-Margulis
measure. Recall \cite{PP} that $\mu$  the measure of maximal entropy for the geodesic flow. It has a full support
in the unit tangent bundle $SM.$ 

\begin{lemma}\cite[Theorem 6.9 and Proposition 7.2]{PP}
\label{LmMor}
$ Lh e^{-Lh} \sum_{L(\gamma)\leq L} \nu_\gamma $ converges as $L\to\infty$ to $\mu.$
\end{lemma}

\begin{lemma}
\label{LmAvoid}
For each $q_0\in M$ there exists $\eps$ such that for each $L$ there is periodic geodesic $\gamma$ such that
$L(\gamma)>L$ and $\gamma$ does not visit an $\eps$ neighborhood of $q_0.$  
\end{lemma}

\begin{proof}[Proof of Lemma \ref{LmEqLen}]
Pick a small $\brdelta$ and large $L.$
By Lemma \ref{LmAvoid} there exists a closed geodesic $\gamma_1$ 
such that $L_g(\gamma_1)>L$ and $d(q(\gamma_1(t)), q_0)>\eps).$ 
Let $\gamma_2$ be a closed geodesic such that $L_g(\gamma_1)<L_g(\gamma_2)<L_g(\gamma_1)+\brdelta$
and $\gamma$ spends at least time $\mu(B(q_0, \eps/2)/2 L_g(\gamma_1)$ inside $B(q_0, \eps/2)$
(the existence of such a geodesic follows from Lemma \ref{LmMor}). Take 
$\tg^\eta=(1-\eta z(q)) g$ where $z(q)=1$ on $B(q_0, \eps/2)$ and $z(q)=0$ outside $B(q_0, \eps).$
We can choose $z$ so that $||z||_{C^r}=O\left(\eps^{-r}\right).$ Then $||g-\tg^\eta||_{C^r}=O(\eta/\eps^r).$ 
Let $\gamma_2^\eta$ be the closed geodesic for $\tg^\eta$ homotopic to $\gamma_2.$ Note that 
$\gamma_1$ is a geodesic of $\tg^\eta$ for each $\eta$ and
$L_{\tg^\eta}(\gamma_1)\equiv L_g(\gamma_1).$ Also
$$L_{\tg^\eta}(\gamma_2^\eta)\leq L_{\tg^\eta}(\gamma_2)\leq L_g(\gamma_1)+\brdelta-\frac{\mu(B(q_0, \eps/2) L_g(\gamma_1) \eta }{2} .$$
Accordingly there exists $\eta<\frac{2\brdelta}{L \mu(B(q_0, \eps/2))}$ such that
$L_{\tg^\eta}(\gamma_1^\eta)=L_{\tg^\eta}(\gamma_2)$ as claimed. 
\end{proof}
In the proof of Lemma \ref{LmAvoid} we need several facts about the dynamics of the geodesic flow which 
we call $\phi_t.$ Recall \cite{Anosov} that  $\phi_t$ is 
uniformly hyperbolic. In particular, there is a cone field $\cK(x)$ and $\lambda>0$ such that for $u\in \cK,$
$||d\phi_t(u)||\geq e^{\lambda t}||u||.$ Moreover the cone field $\cK$ can be chosen in such a way that 
if $x=(q,v)$ and $u=(\delta q, \delta v)\in \cK(x)$ then
\begin{equation}
\label{ConeGeom}
||\delta q||\geq c ||\delta v|| \text{ and }\angle(\delta q, v)\geq\frac{\pi}{4}
\end{equation}
We call a curve $\sigma$ {\it unstable} if $\dot{\sigma}\in\cK.$ By the foregoing discussion if
$\sigma$ is un unstable curve then the length of the projection of $\phi_t(\sigma)$ on $M$ is 
longer than $c e^{\lambda t}.$

\begin{proof}[Proof of Lemma \ref{LmAvoid}] We first show how to construct a not necessary closed geodesic
avoiding $B(q_0, \eps)$ and then upgrade the result to get the existence of a closed geodesics.

The first part of the argument is similar to \cite{BS, D}.
Pick a small $\kappa>0.$ Take an unstable curve $\sigma$ of small length $\kappa.$ 
We show that if $\kappa$ and $\eps$ are sufficiently small then $\sigma$ 
contains a point such that the corresponding geodesic avoids $B(q_0, \eps).$ Let $T_1$ be a number such that
$|\phi_{T_1} (\sigma)|=1$ where $\phi$ denotes the geodesic flow. Note that $T_1=O(|\ln\kappa|).$
Also observe that due to \eqref{ConeGeom} there exists a number $r_0$ such that if $\tsigma$ is an unstable curve and
$x\in \tsigma$ is such that $d(q(x), q_0)<\eps$ then for all $y\in \tsigma$ such that $C\eps\leq\bd(y,x)\leq r_0$
we have
$$\bd(q(\phi_{t} y), q_0)>\eps$$  
for $|t|<r_0$ 
where $\bd$ denotes the distance in the phase space (just take $r_0$ much smaller than the injectivity
radius of $q_0$). 

Thus the set
$$ \{y\in \phi_{T_1}(\sigma): d(q(\phi_{-t} y), q_0)\leq\eps \text{ for some } 0\leq t\leq T_1\} $$
is a union of $O(|\ln\kappa|/r_0)$ components of
length $O(\eps/\kappa^a)$ for some $a>0.$ 
Therefore if $\kappa\ll 1$ and $\eps\ll \kappa$ then the average distance between the components is
much larger than $\kappa.$ So we can find $\sigma_1\subset \phi_{T_1} \sigma$ such that 
$|\sigma_1|=\kappa,$ 
and if  
$y \in \sigma_1$ then $d(q(\phi_{-t} y), q_0)>\eps $ for each $ 0\leq t\leq T_1.$
Take $T_2$ such that $|\phi_{T_2}\sigma_1|=1.$
Then we can find $\sigma_2\subset \phi_{T_2} \sigma_1$ such that
$|\sigma_2|=\kappa,$ 
and if  
$y \in \sigma_2$ then $d(q(\phi_{-t} y), q_0)>\eps $ for each $ 0\leq t\leq T_2.$
We continue this procedure inductively to construct arcs $\sigma_j$ for all $j\in \naturals.$
Taking $$x=\bigcap_{j=1}^\infty \phi_{-(T_1+T_2+\dots+T_j)} \;\; \sigma_j$$ 
we obtain a geodesic avoiding $B(q_0, \eps).$
To complete the proof we need

\begin{lemma}
\label{LmClosing}
(Anosov Closing Lemma) (see \cite[Section 18]{HK})
Given $\eta>0$ there exists $\delta>0$ such that if for some $t_1, t_2$
such that $|t_2-t_1|$ is sufficiently large we have
$\bd(\gamma(t_1), \gamma(t_2)<\delta$ then there exists a closed geodesic $\tgamma$ such that
$|L(\tgamma)-|t_2-t_1||<\eta$ and for each $t\in[t_1, t_2]$ there exists $s$ such that
$\bd(\gamma(t), \tgamma(s))<\eta.$
\end{lemma}
Take $\delta$ corresponding to $\eta=\eps/2.$ Consider points $\gamma(j L)$ where $j=1\dots K.$
By pigeonhole principle if $K$ is sufficiently large we can find $j_1, j_2$ such that
$\bd(\gamma(j_1 L), \gamma(j_2 L))<\delta$ and so there exists a closed geodesic $\tgamma$ avoiding
$B(q_0, \eps/2).$ Since $\eps$ is arbitrary, Lemma \ref{LmAvoid} follows.
\end{proof}

Suppose now that $\dim(M)=2.$ Let $\cH_r(M)$ denote the space of $C^r$ metrics with positive topological entropy.
This set is $C^r$ open (\cite{K}) and dense. (If genus$(M)\geq 2$ then every metric has positive topological 
entropy \cite{K}. For torus the density of $\cH_r(M)$ follows from \cite{Ban} and for sphere it follows
from \cite{KW}).

\begin{theorem}
\label{ThSmallGapDim2}
The set of metrics satisfying \eqref{smallgap} is topologically generic in $\cH(M).$
\end{theorem} 

\begin{corollary}
The set of metrics satisfying \eqref{smallgap} is topologically generic in the space of all metrics on
$M.$
\end{corollary}

\begin{proof}[Proof of Theorem \ref{ThSmallGapDim2}] By \cite{K} if $g\in \cH_r(M)$ then there is a hyperbolic
basic set $\Lambda$ for the geodesic flow. Since Lemmas \ref{LmEqLen}, \ref{LmMor}, \ref{LmAvoid} and
\ref{LmClosing} remain valid in the setting of hyperbolic
sets the proof is similar to the proof of Theorem \ref{smallgap}.
(In the proof of Lemma \ref{LmAvoid} we need to take $\sigma_1$ so that it crosses completely 
an element of some Markov partition $\Pi$ such that all elements of $\Pi$ have unstable length between 
$\kappa$ and $C\kappa.$ The number of eligible segments now is not $O(1/\kappa)$ but 
$O(1/\kappa^a)$ for some $a>0$ but this is still much larger than $|\ln\kappa|.$)
\end{proof}

\section{Small gaps for hyperbolic surfaces, continued}
\label{ScH2-KR}

Here we show that for Lebesgue-typical hyperbolic surface the gaps in the length spectrum cannot be too small.
Our argument in similar to \cite{KR}. 
Related results are obtained in \cite{Varju}.

\subsection{Small values of polynomials.}

\begin{prop} 
\label{PrChExt}
(see e.g \cite[Section 3.2]{MH})
Consider a degree $D$ polynomial $P(x)=a_D x^D+a_{D-1} x^{D-1}+\dots a_0.$ Then
$$ \sup_{[-1, 1]} |P(x)|\geq \frac{|a_D|}{2^{D-1}}. $$
\end{prop} 

\begin{cor}\label{CrMultiCheb}
Let $0\neq P\in\integers[x_1, x_2\dots x_n],$ $\deg(P)=D$ then
$$ \sup_{[-1, 1]^n} |P(x)|\geq \frac{1}{2^{D-1}}. $$ 
\end{cor}

\begin{proof}
By induction. For $n=0$ or $1$ the result follows from Proposition \ref{PrChExt}.

Next, suppose the statement is proven for polynomials of $n-1$ variables. If $P$ does not depend
on $x_n$ then we are done. Otherwise
let $k>0$ be the degree of $P$ with respect to $x_n$. Then 
$$ P(x)=a_k(x_1, \dots, x_{n-1}) x_n^k+a_{k-1}(x_1,\dots, x_{n-1}) x_n^{k-1}+\dots+a_0(x_1,\dots, x_{n-1}) $$
where $a_k$ is the polynomial with integer coefficients of degree $D-k.$ Let
$$(\brx_1, \dots \brx_{n-1})=\arg\max_{[-1, 1]^{n-1}} |a_k(x_1, \dots, x_{n-1})|. $$ 
Then
$$ \sup_{[-1, 1]^n} |P(x_1,\dots x_{n-1}, x_n)|\geq \max_{x_n\in [-1, 1]}
|P(\brx_1,\dots \brx_{n-1}, x_n)|\geq |a(\brx_1, \dots,\brx_{n-1})| 2^{1-k}$$
$$\geq
2^{1+k-D} 2^{1-k}=2^{2-D} $$
completing the proof.
\end{proof}

\begin{prop}
\label{PrRemez}
{\sc (Remez inequality)} (see \cite{BG} or \cite[Theorem 1.1]{Y})
Let $B$ be a convex set in $\reals^n,$ $\Omega\subset B,$ and $P$ be a polynomial 
of degree $D.$ Then
$$ \sup_{B}|P|\leq C_B {\rm mes}^{-D} (\Omega) \sup_{\Omega} |P|. $$
\end{prop}

\begin{cor} \label{CrRemez}
Under the conditions of Proposition \ref{PrRemez}
$$ \mes(x\in B: |P(x)|\leq \eps)\leq \left(\frac{C_B \eps}{\sup_B |P|}\right)^{1/D}. $$
\end{cor}

\begin{proof}
Apply Proposition \ref{PrRemez} with $\Omega=\{x\in B: |P(x)|\leq \eps\}.$
\end{proof}

\begin{cor}
\label{CrDiop}
If $P_N\in \integers[x_1, x_2, \dots, x_n]$ are polynomials of degree $D_N$ and $\eps_N$ is a sequnces such that
$\displaystyle \sum_N \eps^{1/D_N}<\infty $ then $|P(x_1, \dots x_n)|<\eps_N$ has 
only finitely many solutions for almost every $(x_1\dots x_n)\in \reals^n.$
\end{cor}

\begin{proof}
It suffices to show this for a fixed cube $B$ with side 2. Then Corollaries \ref{CrMultiCheb} and 
\ref{CrRemez} give
$$ \mes(x\in B: |P_N(x)|\leq \eps_N)\leq \left(C 2^{D_N} \eps_N\right)^{1/D_N}
=\bar{C} \eps_N^{1/D_N}$$
so the statement follows from Borel-Cantelli Lemma.
\end{proof}

\subsection{Polynomial maps on $SL_2(\reals)$}

\begin{cor} 
\label{CrDiopSL2}
Let $m$ be a fixed number.

(a) Let
$P_N\in \integers((a_1, b_1, c_1, d_1),\dots, (a_m, b_m, c_m, d_m))$ be polynomials of degree $D_N.$ 
For $A_1, \dots A_m\in SL_2(\reals)$ with 
$A_j=\left(\begin{array}{cc} a_j & b_j \\ c_j & d_j\\\end{array}\right) $ let
$$ H_N(A_1, \dots, A_m)=P_N((a_1, b_1, c_1, d_1),\dots, (a_m, b_m, c_m, d_m)). $$
If $\displaystyle \sum_N \eps_N^{1/((m+2) D_N)}<\infty$ then
$ |H_N(A_1\dots A_m)|<\eps_N $
for only finitely many $N$ for almost every $(A_1,\dots A_m)\in (SL_2(\reals))^m. $

(b) Given  $g\in\naturals$ let 
$$G_g=\{(A_1, \dots A_{2g})\in (SL_2(\reals))^{2g}: 
[A_1, A_2] [A_3, A_4]\dots [A_{2g-1}, A_{2g}]=I\}. $$
Let $P_N\in \integers((a_1, b_1, c_1, d_1),\dots, (a_{2m}, b_{2m}, c_{2m}, d_{2m}))$ 
be polynomials of degree $D_N.$ Let
$$ H_N(A_1, \dots, A_{2g})=P_N((a_1, b_1, c_1, d_1),\dots, (a_{2g}, b_{2g}, c_{2g}, d_{2g})). $$
Assume that $H_N$ is not identically equal to 0 on $G_g.$
If 
$\displaystyle \sum_N \eps_N^{\delta_N}<\infty$
where $\delta_N=\frac{1}{(4g-2)(g+2) D_N}$ then
$ |H_N(A_1\dots A_{2g})|<\eps_N $
for only finitely many $N$ for almost every $(A_1,\dots A_{2g})\in G_g. $
\end{cor}

\begin{proof}
(a) It suffices to prove the statement under the assumption that $|a_j|>\delta$ for some fixed $\delta>0.$ Then
$d_j=\frac{1+b_j c_j}{a_j}$ and so
$$ H(A_1,\dots A_m)=\frac{\tP_N((a_1, b_1, c_1),\dots (a_m, b_m, c_m))}{\prod_{j=1}^m a_j^{d_N}} $$
where $\tP_N$ is a polynomial of degree $\tD_n\leq (m+2) d_N.$ Thus if $|P_N|\leq \eps_N$ then
$|\tP_N|\leq \teps_N:=\frac{\eps_N}{\delta^{(m+2) D_N}}.$ Since
$$\sum_N \teps_N^{1/\tD_N}\leq \frac{1}{\delta} \sum_N \eps_N^{1/(m+2) D_N}<\infty $$
the result follows from Corollary \ref{CrDiop}. 

(b) Rewriting the equations defining $G_g$ in the form
$$ [A_1, A_2]\dots [A_{2g-3}, A_{2g-2}] A_{2g-1} A_{2g} A_{2g-1}^{-1}=A_{2g} $$
we can express the entries of $A_{2g}$ as rational functions of the entries of the other matrices.
Arguing as in part (a) we can reduce the inequality 
$|P_N(A_1,\dots A_{2g-1}, A_{2g})|<\eps$ to
$|\hP_N(A_1,\dots A_{2g-1})|<\heps_N$ where $\hP_N$ is the polynomial of degree $(4g-2) D_N.$
Now the result follows from part (a). 
\end{proof}

\begin{cor}
For each $\eta>0$ for almost every $A_1, \dots A_m\in SL_2(\reals)$ the inequality
$$ ||W(A_1, \dots A_m)-I||>(2m-1)^{-|W|^2 (m+2+\eta)} $$
holds for all except for finitely many words $W.$
\end{cor}

\begin{proof}
If $||W(A_1, \dots A_m)-I||\leq \eps$ then all entries of $W-I$ are $\eps$ close to I. Conisdering for example, 
the condition $W_{11}(A_1, \dots A_m)-1$ we 
get a polynomial of degree $|W|.$ Therefore, by Corollary \ref{CrDiopSL2} it suffices to check that
$$ \sum_W (2m-1)^{-\frac{|W|^2 (m+2+\eta)}{|W|(m+2)}} <\infty $$
but the above sum equals to
$$ \sum_D (2m-1)^D (2m-1)^{-D-D(\eta/(m+2)}=\sum_D (2m-1)^{-\eta D/(m+2)}<\infty. \qedhere $$
\end{proof}

\begin{cor}
For $\cA=(A_1\dots A_{2g})\in G_g$ let $S_\cA$ be the surface defined by $\cA.$ Given
a word $W$ let $l(W, \cA)$ be the length of the closed geodesic in the homotopy class defined by $W.$
Then for each $\eta>0$ the following holds for almost all $\cA\in G_g$

There exists a constant  $K=K(\cA)$ such that for each pair $W_1, W_2$ either 
$$l(W_1, \cA)=l(W_2, \cA)\text{ or } $$
\begin{equation}
\label{QuadExp}
|l(W_1, \cA)-l(W_2, \cA)|\geq K (4g-1)^{-[(2g+4)(4g-2)+\eta] \max^2(|W_1|, |W_2|)}. 
\end{equation}
\end{cor}

\begin{remark}
  Recall that \cite{Randol1} shows that for any hyperbolic surface the length spectrum has unbounded multiplicity
  so there are many pairs of non conjugated words there the first alternative of the corollary holds.
\end{remark}  

\begin{remark}
Note that \eqref{EqWL-GL} shows that $l(W_1, \cA)$ can be close to $l(W_2, \cA)$ only if the lengths
of $W_1$ and $W_2$ are of the same order. Thus \eqref{QuadExp} implies that for almost every $\cA$ there are
constants $K,$ $R$ such that
$$ |l(W_1, \cA)-l(W_2, \cA)|\geq K e^{-R(\cA) l^2(W_1, \cA)}. $$
\end{remark}

\begin{proof}
Let $P_{W}(\cA)=\tr(W(\cA)).$ Since $P_W(\cA)=2\cosh(l(W,\cA)/2),$ it follows that if $l(W_1(\cA))$ is 
close to $l(W_2(\cA))$ then
$$l(W_1, \cA)-l(W_2, \cA)|\geq  C |P_{W_1}(\cA)-P_{W_2}(\cA)| e^{-|W_1|D} .$$
Therefore it suffices to show that if $l(W_1,\cA)\neq l(W_2, \cA)|$ then
$$ |P_{W_1}(\cA)-P_{W_2}(\cA)|\geq \tK e^{-|W_1|D} (4g-1)^{-[(2g+4)(4g-2)+\eta] \max^2(|W_1|, |W_2|)}. $$
Since $\eta$ is arbitrary, we can actually check that
$$ |P_{W_1}(\cA)-P_{W_2}(\cA)|\geq K (4g-1)^{-[(2g+4)(4g-2)+\eta] \max^2(|W_1|, |W_2|)} .$$
To verify this we will show that for almost all $\cA\in G_m$ the inequality
$$ |P_{W_1}(\cA)-P_{W_2}(\cA)|<  (4g-1)^{-[(2g+4)(4g-2)+\eta] \max^2(|W_1|, |W_2|)} $$
has only finitely many solutions. Let 
$P_{W_1, W_2}(\cA)=P_{W_1}(\cA)-P_{W_2}(\cA).$ It is a polynomail of degree 
$\max(|W_1|, |W_2|).$ So by Corollary \ref{CrDiopSL2}(b) it suffices to check that
$$ \sum_{W_1, W_2} (4g-1)^{-\frac{[(2g+4)(4g-2)+\eta] \max(|W_1|, |W_2|)}{(4g-2)(g+2)}}<\infty $$
There are at most $(4g-1)^{2k}$ pairs $(W_1, W_2)$ with $k=\max(W_1, W_2)$ so the last sum is estimated by
$$ \sum_{k} (4g-1)^{2k} (4g-1)^{-\frac{[(2g+4)(4g-2)+\eta]k}{(4g-2)(g+2)}}=
\sum_k (4g-1)^{-\frac{\eta k}{(4g-2)(g+2)}}<\infty $$
proving the result.
\end{proof}

\section{Open problems.}
(1) A suitable version of Theorem \ref{ThAlgGen} should hold for other symmetric spaces.
In particular, recall that arithmetic manifolds appear as fundamental 
domains $G/\Gamma$ where $G$ is a connected semi-simple algebraic 
$\reals$-group without compact factors of $\reals$-rank $\geq 2$, and $\Gamma$ 
is a lattice in $G$ (cf. \cite{Margulis74,Margulis75,Margulis77}). 
Thus we expect that a version of Theorem
\ref{ThAlgGen} should hold in higher rank setting. Note however, that for higher rank symmetric spaces
closed orbits are not
isolated but appear in families.

(2) The proof of Theorem \ref{thm:smallgapneg} relies on localized perturbations. Therefore it does not work in the
analytic category. We expect that Theorem \ref{thm:smallgapneg} is still valid for analytic metrics but the proof
would require new ideas.

(3) It is likely that an explicit lower bound for the gaps in the length spectrum could also be obtained
for {\em prevalent} set of negatively curved metrics (see \cite{Kal} for related results) but we do not
pursue this question here.

\medskip
\centerline{\bf Acknowledgements.}   \smallskip

The authors started discussing questions about gaps in the length spectra in 2005 when both 
were participating in the work of the thematic program ``Time at work'' at the Institut Henri 
Poincar\'e in 2005.  

The authors thank D. Popov for stimulating their interest in this problem, and 
Y. Yomdin and V. Kaloshin for discussions related to
Sections \ref{ScH2-KR} of the present paper.  The authors would also like to thank 
A. Glutsuk, N. Kamran, A. Katok, P. Sarnak and L. Silberman for stimulating discussions.


\end{document}